\documentclass[notitlepage]{amsart}

\usepackage{amsfonts}
\usepackage[latin1]{inputenc}
\usepackage[all]{xy}
\usepackage{graphicx}

\newtheorem{theorem}{Theorem}[subsection]
\newtheorem{corollary}[theorem]{Corollary}

\newtheorem{example}[theorem]{Example}
\newtheorem{examples}[theorem]{Examples}
\newtheorem{lemma}[theorem]{Lemma}

\newtheorem{proposition}[theorem]{Proposition}
\newtheorem{conjecture}[theorem]{Conjecture}
\newtheorem{remark}[theorem]{Remark}

\newcommand{\pn}{\par\noindent}

\newcommand{\mor}[3]{$\xymatrix@1@C=15pt{#3: #1\ar[r]& #2}$}
\newcommand{\iso}[3]{$\xymatrix@1@C=15pt{#3: #1\ar[r]^-{\cong}& #2}$}

\newcommand{\Hom}{{\rm Hom\,}}
\newcommand{\Ext}[1]{{\rm Ext}#1}
\newcommand{\Tor}[1]{{\rm Tor}#1}

\newcommand{\Ker}{{\rm Ker}\,}
\newcommand{\Coker}{{\rm Coker}\,}

\renewcommand{\top}{{\rm top}\,}

\newcommand{\und}[1]{S(#1)}
\newcommand{\ov}[1]{Q(#1)}
\newcommand{\findim}{{\rm fin.dim.\, }}
\renewcommand{\P}{\mathcal{P}^{<\infty}}

\renewcommand{\mod}{{\rm mod\,}}
\newcommand{\pd}{{\rm pd\,}}
\newcommand{\id}{{\rm id\,}}

\newcommand{\End}{{\rm End\,}}
\newcommand{\C}{\mathcal{C}}

\newcommand{\X}{\mathcal{X}}
\newcommand{\Y}{\mathcal{Y}}
\newcommand{\F}{\mathcal{F}}

\newcommand{\T}{\mathcal{T}}

\newcommand{\CS}{\und{\mathcal{C}}}
\newcommand{\CQ}{\ov{\mathcal{C}}}

\usepackage{latexsym,amssymb,amscd}
\usepackage{amsmath}
\usepackage[all]{xy}

\begin{document}
\sloppy
\bibliographystyle{plain}

\title[Filtering subcategories of modules of an artinian algebra]{Filtering subcategories of modules of an artinian algebra}

\author{F. Huard, M. Lanzilotta and D. Smith}
\address{\pn Fran\c cois Huard; Department of mathematics,
Bishop's University. Sherbrooke, Qu\'ebec, Canada,  J1M1Z7.}
\email{fhuard@ubishops.ca}
\address{\pn Marcelo Lanzilotta; Centro de Matem\'atica (CMAT),
Instituto de Matem\'atica y Estad\'istica Rafael Laguardia (IMERL),
Universidad de la Rep\'ublica,
Igu\'a 4225, C.P. 11400, Montevideo, Uruguay.}
\email{marclan@cmat.edu.uy, marclan@fing.edu.uy}
\address{\pn David Smith; Department of mathematics,
Bishop's University. Sherbrooke, Qu\'ebec, Canada,  J1M1Z7.}
\email{dsmith@ubishops.ca}

\begin{abstract} 
Let $A$ be an artinian algebra, and let $\mathcal{C}$ be a subcategory of mod$A$ that is closed under extensions. When $\mathcal{C}$ is closed under kernels of epimorphisms (or closed under cokernels of monomorphisms), we describe the smallest class of modules that filters $\mathcal{C}$.  As a consequence, we obtain sufficient conditions for the finitistic dimension of an algebra over a field to be finite.  We also apply our results to the torsion pairs. In particular, when a torsion pair is induced by a tilting module, we show that the  smallest classes of modules that filter the torsion and torsion-free classes are completely compatible with the quasi-equivalences of Brenner and Butler.
\end{abstract}

\maketitle

\let\thefootnote\relax\footnote{The third author was supported by a discovery grant from NSERC of Canada.}

\section*{Introduction}

Let $A$ be an artin algebra and let $\mod A$ be the category of finitely generated $A$-modules. In \cite{Z95}, B. Zimmermann Huisgen asked the following question: \lq\lq What is the structure of module in $\mod A$ having finite projective dimension, in contrast to the structure of those of infinite projective dimension?\rq\rq\,  In the same paper, the author pointed out that a promising key to such structure theorems was provided by Auslander and Reiten for the case where $\P$ is contravariantly finite.

The result of Auslander and Reiten \cite[(3.8)]{AR91} asserts that if $\X$ is a resolving and contravariantly finite subcategory of $\mod A$, then the minimal right $\X$-approximations $f_i:X_i\rightarrow S_i$ of the simple $A$-modules have the following property : a module $M$ in $\mod A$ is in $\X$ if and only if it is a direct summand of a $\{X_1, X_2, \dots, X_n\}$-filtered module, that is, $M$ is a direct summand of an $A$-module $\widetilde{M}$ for which there exists a sequence of submodules $0=\widetilde{M_0}\subset \widetilde{M_1}\subset \cdots \subset \widetilde{M_n}=\widetilde{M}$ such that $\widetilde{M_i}/\widetilde{M_{i-1}}\in \{X_1, X_2, \dots, X_n\}$ for all $i=1, 2, \dots, n$.  Since then, the minimal right $\P$-approximations of the simple modules were commonly considered as the basic building blocks for $\P$, in case $\P$ is contravariantly finite in $\mod A$. However, this approach has two major obstacles.  Firstly, for few classes of algebras is it known whether $\P$ is contravariantly finite.  Secondly, when $\P$ is contravariantly finite and $M\in\P$, it is in practice often challenging to find $\widetilde{M}$ and its $\{X_1, X_2, \dots, X_n\}$-filtration. Consequently, the right minimal $\P$-approximations of the simples are more theoretical building blocks than practical building blocks.  Our first goal was to fix these issues. However, our approach applies to more general situations.

Let $A$ be an artinian algebra, and let $\C$ be a subcategory of $\mod A$ that is closed under extensions and closed under kernels of epimorphisms (or closed under cokernels of monomorphisms).  In this paper, we describe the smallest class of modules, denoted by $\CQ$ (or $\CS$) (see Subsection \ref{Definitions and Elementary Properties} for details), having the property that a given module $M$ is in $\C$ if and only if $M$ has a ${\CQ}$-filtration (or a ${\CS}$-filtration, respectively).  In the case where $\C$ is closed under kernels of epimorphisms and cokernels of monomorphisms in $\mod A$, we have ${\CQ}={\CS}$. In particular, $\ov{\P}=\und{\P}$. Consequently, our approach does not rely on the contravariantly finiteness of $\P$ (or, more generally, of $\C$). Moreover, when ${\CQ}$ is known and $M\in \C$, it is generally easy to construct the ${\CQ}$-filtration of $M$, and similarly for ${\CS}$. The drawback of this approach however lies on the following catch: the classes ${\CQ}$ and ${\CS}$ are generally infinite, and their elements are often difficult to track.

Our results apply to several situations, as encountering subcategories that are closed under extensions and kernels of epimorphisms (or cokernels of monomorphisms) is frequent in representation theory of algebras.  For instance, in addition to $\P$, our results apply to any subcategory of $\mod A$ that is submodule-closed or quotient-closed and to the classes involved in any torsion or cotorsion pair.  

The ideas contained in this paper are elementary.  Nevertheless, to our knowledge, nothing has ever been written along these lines.  Only traces of the elements in $\ov{\P}=\und{\P}$ were found in works of D. Happel \cite{Hap} and J. \v{S}t'ov\'i\v{c}ek \cite{Sto04}, and more recently in the Ph.D. thesis of G. Mata \cite{Mata15}.

Our paper is organized as follows.  In Section 1, we give the necessary definitions and prove the main theorem of this paper.

\vspace{0.1cm}
\noindent\textbf{Theorem} (Theorem \ref{Thm 1.5})\textbf{.}  Let $A$ be an artinian algebra and let $\C$ be an extension-closed subcategory of $\mod A$.
\begin{enumerate}
\item[(a)] If $\C$ is closed under kernels of epimorphisms, then $\CQ$ is the smallest class of modules that filters $\C$.
\item[(b)] If $\C$ is closed under cokernels of monomorphisms, then $\CS$ is the smallest class of modules that filters $\C$. 
\end{enumerate}

\noindent  As a nice consequence, we obtain a proof of the following result concerning the finitistic dimension of $A$, denoted by $\findim A$. Observe that this result, and the idea of the proof, first appeared in an unpublished work by Happel~\cite{Hap}.  However, the use of a result by Jensen and Lenzing~\cite{JL89}, that was apparently unknown to Happel, allows us to provide a slightly more detailed and rigorous proof. 

\vspace{0.1cm}
\noindent\textbf{Theorem} (Theorem \ref{Thm Happel})\textbf{.}
Let $A$ be a finite-dimensional algebra over a field $k$. If the length of every indecomposable module in $\und{\P}$ is bounded, then $\findim A<\infty$.
\vspace{0.1cm}

We also show the non-uniqueness of the $\CQ$- or $\CS$-filtrations, showing that the Jordan-H\"older condition does not hold in this context. We also give some remarks on the position of the elements of $\CQ$ and $\CS$ in the Auslander-Reiten quiver of $\mod A$.

Section 2 is devoted to the study of inclusions of subcategories satisfying some conditions.  In particular, we show in Example \ref{ex findim} that $\findim A\leq n$ if and only if $\und{\mathcal{P}^{\leq n+1}}\subseteq \mathcal{P}^{\leq n}$, where $\mathcal{P}^{\leq n}$ stands for the class of all finitely generated $A$-modules of projective dimension at most $n$.


Finally, Section 3 is devoted to the study of the torsion and torsion-free classes of any torsion pair.  In particular, when a torsion pair $(\mathcal{T}(T), \mathcal{F}(T))$ is induced by a tilting $A$-module $T$, we show that the our construction is totally compatible with the Brenner-Butler Theorem in the following sense:

\vspace{0.1cm}
\noindent\textbf{Theorem} (Theorem \ref{Prop 3.8})\textbf{.} 
Let $T$ be a tilting $A$-module and $B=\End_AT$.  Let  $(\mathcal{T}, \mathcal{F})$ and  $(\mathcal{X}, \mathcal{Y})$ be the torsion pairs induced by $T$ in $\mod A$ and $\mod B$ respectively.
\begin{enumerate}
\item[(a)] If an $A$-module $M$ is $\und{\T}$-(co)filtered, then the $B$-module $\Hom_A(T, M)$ is $\ov{\Y}$-(co)filtered.
\item[(b)] If a $B$-module $M$ is $\ov{\Y}$-(co)filtered, then the $A$-module $M\otimes_B T$ is $\und{\T}$-(co)filtered.  
\item[(c)] If an $A$-module $M$ is $\ov{\F}$-(co)filtered, then the $B$-module $\Ext^1_A(T, M)$ is $\und{\X}$-(co)filtered.
\item[(d)] If a $B$-module $M$ is $\und{\X}$-(co)filtered, then the $A$-module $\Tor^B_1(M, T)$ is $\ov{\F}$-(co)filtered. 
\end{enumerate}
\vspace{0.1cm}

Our work is completed with a discussion on the modules having (proper) standard or (proper) costandard filtrations, in connection with quasi-hereditary and standardly stratified algebras.

\section{Filtering subcategories of modules that are closed under extensions}

In this paper, $A$ will stand for an artinian algebra and we will denote by $\mod A$ the category of finitely generated $A$-modules. Given $M\in\mod A$, we will denote by $\pd  M$ its projective dimension and by $\ell(M)$ its \textbf{length}, that is the number of simple $A$-modules appearing in any composition series of $M$. Finally, $\C$ will be a nonzero subcategory of $\mod A$, that is a subcategory of $\mod A$ containing at least one nonzero module. Also, we will use the symbols \lq\lq$\subset$\rq\rq \, to denote a strict inclusion and \lq\lq$\subseteq$\rq\rq \, to denote a possibly non-strict inclusion.

\subsection{Definitions and elementary properties}\label{Definitions and Elementary Properties}

Let $A$ be an artinian algebra and $\mathcal{C}$ be a nonzero subcategory of $\mod A$. Let
\[
\CS =\{ M\in \C \ | \ M\neq 0 \text{ and } 0\neq L \subset M \text{ implies } L\notin \C\}
\]
and
\[
\CQ =\{ M\in \C \ | \ M\neq 0 \text{ and } 0\neq L \subset M \text{ implies } M/L\notin \C\}.
\]
In other words, $\CS$ consists of the nonzero modules in $\C$ whose proper nonzero submodules are not in $\C$, while $\CQ$ consists of the nonzero modules in $\C$ whose proper nonzero quotients are not in $\C$. These classes of modules will play a determinant role in the sequel. We start by proving of elementary properties.

\begin{proposition}\label{Lem 1.6}
Let $\C$ be a nonzero subcategory of $\mod A$.
If $M$ is a nonzero module of minimal (positive) length in $\C$, then $M\in \CS\cap \CQ$.
\end{proposition}
\begin{proof}
Suppose that $M$ is a nonzero module of minimal length in $\C$.  If $L$ is a proper nonzero submodule of $M$, then $\ell(L)<\ell(M)$ and $\ell(M/L)<\ell(M)$.  The minimality of $M$ gives $L\notin \C$ and  $M/L\notin \C$, so $M\in \CS\cap \CQ$.
\end{proof}

\begin{lemma}\label{Lem 3.11} Let $\C$ be a nonzero subcategory of $\mod A$.
\begin{enumerate}
\item[(a)] If $\C$ is closed under submodules, then $\CS=\{\text{simple $A$-module in }\C\}$.
\item[(b)] If $\C$ is closed under quotients, then $\CQ=\{\text{simple $A$-module in }\C\}$.
\end{enumerate}
\end{lemma}
\begin{proof}
This is clear since every nonzero module has a simple submodule and a simple quotient, and every simple module in $\C$ is in $\CQ\cup\CS$ by (\ref{Lem 1.6}).
\end{proof}

In the particular case where we take $\C=\mod A$, the above proposition tells that ${\CS}=\{\text{simple modules in }\mod A\}={\CQ}$, while in general ${\CS}\neq {\CQ}$.

The following observation indicates that, under a very mild hypothesis, the modules in $\CS$ and $\CQ$ are indecomposable.

\begin{proposition}\label{Rem 1.1 (c)}
Let $\C$ be a nonzero subcategory of $\mod A$.
If $(\mod A)\setminus\C$ is closed under direct sums, then every element in $\CS\cup \CQ$ is indecomposable.
\end{proposition}
\begin{proof}
Suppose that $M\in\CS\cup\CQ$ and $M = M_1\oplus M_2$ for some nonzero $A$-modules $M_1$ and $M_2$.  Since $M_1$ and $M_2$ are submodules of $M$, and $M_1\cong M/M_2$ and $M_2\cong M/M_1$, it follows from the definitions of $\CS$ and $\CQ$ that  $M_1\notin \C$ and $M_2\notin \C$. Since $(\mod A)\setminus\C$ is closed under direct sums, we get that $M=M_1\oplus M_2$ is not in $\C$, a contradiction.  So $M$ is indecomposable.
\end{proof}

The following lemma will play an important role in the proof of our main theorem.

\begin{lemma}\label{Lem 1.3} Let $\C$ be a nonzero subcategory of $\mod A$. Let $M$ be a nonzero module in $\C$.
\begin{enumerate}
\item[(a)] There exists a submodule $L$ of $M$ such that $L\in \CS$. 
\item[(b)] There exists a submodule $L$ of $M$ such that $M/L\in \CQ$. 
\end{enumerate}
\end{lemma}
\begin{proof}
(a). Let $M$ be a nonzero module in $\C$.  If $M\in \CS$, there is nothing to show.  Else, there exists a proper nonzero submodule $M_1$ of $M$ such that $M_1\in\C$.  If $M_1\in \CS$, we are done. Repeating this process gives a sequence of nonzero submodules
\[
\cdots \subset M_3 \subset M_2 \subset M_1 \subset M
\]
that must stop since $A$ is artinian.  So $M_i\in\CS$ for some $i$.

(b).  Let $M\in\C$.  If $M\in \CQ$, there is nothing to show.  Else, there exists a proper nonzero submodule $M_1$ of $M$ such that $M/M_1\in\C$. If $M/M_1\in \CQ$, we are done. Else, there exists a proper submodule $M_2$ of $M$ containing properly $M_1$ such that $\frac{M/M_1}{M_2/M_1}\cong M/M_2\in \C$. Repeating this process gives a sequence of submodules 
\[
M_1\subset M_2 \subset M_3 \subset \cdots \subset M
\]
that must stop since $A$ is noetherian (since artinian).  So $M/M_i\in\CQ$ for some $i$.
\end{proof}

\subsection{Subcategories closed under kernels of epimorphisms or cokernels of monomorphisms}

The following result gives situations where $\CS$ and $\CQ$ can be compared.

\begin{lemma}\label{Rem 1.1(a)(b)} Let $\C$ be a nonzero subcategory of $\mod A$.
\begin{enumerate}
\item[(a)] If $\C$ is closed under kernels of epimorphisms, then $\CS\subseteq \CQ$.
\item[(b)] If $\C$ is closed under cokernels of monomorphisms, then $\CQ\subseteq \CS$.
\item[(c)] If $\C$ is closed under kernels of epimorphisms and cokernels of monomorphisms, then $\CQ= \CS$.
\end{enumerate}
\end{lemma}
\begin{proof}
We only prove (a); the proof of (b) is dual while (c) clearly follows from (a) and (b).  Let $M\in\CS$. If $M\notin \CQ$, then there exists a proper nonzero submodule $L$ of $M$ such that $M/L\in\C$. Since $L$ is the kernel of the canonical epimorphism $M\rightarrow M/L$, it follows from the hypothesis that $L\in\C$, a contradiction to $M\in\CS$.  So $M\in \CQ$.
\end{proof}

\begin{corollary}\label{Lem 1.4} Let $\C$ be a nonzero subcategory of $\mod A$. Let $M\in \C$.
\begin{enumerate}
\item[(a)] If $\C$ is closed under kernels of epimorphisms, then there exists a submodule $L$ of $M$ such that $L\in \CQ$. 
\item[(b)] If $\C$ is closed under cokernels of monomorphisms, then there exists a submodule $L$ of $M$ such that $M/L\in \CS$. 
\end{enumerate}
\end{corollary}
\begin{proof}
This follows directly from (\ref{Lem 1.3}) and (\ref{Rem 1.1(a)(b)}).
\end{proof}

The following result gathers interesting observations about the Ext and Tor functors.

\begin{proposition}\label{Prop 1.9} Let $\C$ be a nonzero subcategory of $\mod A$, and let $M\in \mod A$. 
\begin{enumerate}
\item[(a)] If $\C$ is closed under kernels of epimorphisms, then for all integer $n\geq 0$ we have 
\begin{enumerate}
\item[(i)] $\Ext^n_A(-,M)\mid_{\C}=0$ if and only if $\Ext^n_A(-,M)\mid_{\CQ}=0$.
\item[(ii)] $\Ext^n_A(M, -)\mid_{\C}=0$ if and only if $\Ext^n_A(M, -)\mid_{\CQ}=0$.
\item[(iii)] $\Tor_n^A(-,M)\mid_{\C}=0$ if and only if $\Tor_n^A(-,M)\mid_{\CQ}=0$.
\item[(iv)] $\Tor_n^A(M, -)\mid_{\C}=0$ if and only if $\Tor_n^A(M, -)\mid_{\CQ}=0$.
\end{enumerate}

\item[(b)] If $\C$ is closed under cokernels of monomorphisms, then for all integer $n\geq 0$ we have  
\begin{enumerate}
\item[(i)] $\Ext^n_A(-,M)\mid_{\C}=0$ if and only if $\Ext^n_A(-,M)\mid_{\CS}=0$.
\item[(ii)] $\Ext^n_A(M, -)\mid_{\C}=0$ if and only if $\Ext^n_A(M, -)\mid_{\CS}=0$.
\item[(iii)] $\Tor_n^A(-,M)\mid_{\C}=0$ if and only if $\Tor_n^A(-,M)\mid_{\CS}=0$.
\item[(iv)] $\Tor_n^A(M, -)\mid_{\C}=0$ if and only if $\Tor_n^A(M, -)\mid_{\CS}=0$.
\end{enumerate}\end{enumerate}
\end{proposition}
\begin{proof}
We only prove (a)(i); the other proofs can be obtained by analogue arguments, sometimes using $\Tor$ instead of $\Ext$. The necessity is obvious since $\CQ\subseteq \C$.  For the sufficiency, we will show that if $\Ext^n_A(-,M)\mid_{\C}\neq 0$ then $\Ext^n_A(-,M)\mid_{\CQ}\neq 0$.  Suppose that $\Ext^n_A(-,M)\mid_{\C}\neq 0$. Let $N$ be a module of minimal length for the following properties: $N\in \C$ and  $\Ext^n_A(N,M)\neq 0$.  To prove the statement, it is sufficient to show that  $N\in\CQ$. Clearly, $N\neq 0$ since $\Ext^n_A(N,M)\neq 0$. Now, let $L$ be a proper nonzero submodule of $N$ and suppose, for sake of contradiction, that $N/L\in \C$.  Since $\C$ is closed under kernels of epimorphisms, we also have $L\in\C$. Moreover, the minimality of $N$ gives $\Ext^n_A(L,M)= 0$ and $\Ext^n_A(N/L,M)= 0$. But the long exact sequence of homology
\[
\cdots\rightarrow \Ext_A^n(N/L, M)\rightarrow \Ext_A^n(N, M)\rightarrow\Ext_A^n(L, M)\rightarrow \cdots
\]
gives $\Ext_A^n(N, M)=0$, a contradiction. So $N/L\notin \C$ and $N\in \CQ$.
\end{proof}

\subsection{Subcategories closed under extensions}

Let $A$ be an artinian algebra and $\X$ be a class of modules in $\mod A$.
We say that an $A$-module $M$ is \textbf{filtered by $\X$} if there exists a sequence of proper inclusions of submodules of $M$
\[
0=M_0\subset M_1 \subset \cdots \subset M_{n-1}\subset M_n=M
\]
 such that $M_i/M_{i-1}\in \X$ for all $i\in\{1, 2, \dots, n\}$. Such a sequence is called a \textbf{$\X$-filtration of $M$}.  We denote by $\mathcal{F}(\X)$ the set of all $\X$-filtered modules. Finally, if $\C$ is a subcategory of $\mod A$, we say that \textbf{$\X$ filters $\C$}, and we write $\F(\X)=\C$ if every nonzero module in $\C$ is $\X$-filtered.

Similarly, we say that an $A$-module $M$ is \textbf{cofiltered by $\X$} if there exists a sequence of proper epimorphisms 
\[
\SelectTips{eu}{10}\xymatrix{M=M_n\ar[r]^{f_n}&M_{n-1} \ar[r]^{f_{n-1}} & \cdots \ar[r]^{f_2}& M_1 \ar[r]^-{f_1} & M_0=0}
\]
such that \Ker$f_i\in\X$ for all $i\in\{1, 2, \dots, n\}$. Such a sequence is called a \textbf{$\X$-cofiltration of $M$}.  We denote by $co\mathcal{F}(\X)$ the set of all $\X$-cofiltered modules. Finally, if $\C$ is a subcategory of $\mod A$, we say that \textbf{$\X$ cofilters $\C$}, and we write $co\F(\X)=\C$, if every nonzero module in $\C$ is $\X$-cofiltered.

The following proposition will be useful in the sequel.  Its easy verification is left to the reader.

\begin{proposition}\label{Prop 1.7} 
Let $\C$ be an extension-closed subcategory of $\mod A$, and $\X$ be a class of modules in $\C$. Then $\mathcal{F}(\X)=co\mathcal{F}(\X)$.  More precisely, 
\begin{enumerate}
\item[(a)] if $M\in \mod A$ and \[
0=M_0\subset M_1 \subset \cdots \subset M_{n-1}\subset M_n=M
\]
is a $\X$-filtration of $M$, then the induced sequence of epimorphisms
\[
\SelectTips{eu}{10}\xymatrix{M=M/M_0\ar[r]^-{f_1}&M/M_{1} \ar[r]^{f_{2}} & \cdots \ar[r]^-{f_{n-2}}& M/M_{n-1} \ar[r]^-{f_{n-1}} & M/M_n=0}
\]
is a $\X$-cofiltration of $M$. Moreover, $\Ker f_i\cong M_i/M_{i-1}$ for all $i\in\{1, 2, \dots, n-1\}$.

\item[(b)] if $M\in \mod A$ and \[
\SelectTips{eu}{10}\xymatrix{M=M_n\ar[r]^{f_n}&M_{n-1} \ar[r]^{f_{n-1}} & \cdots \ar[r]^{f_2}& M_1 \ar[r]^-{f_1} & M_0=0}
\]
is a $\X$-cofiltration of $M$, then the induced sequence of inclusions
\[
0\subset \Ker(f_n) \subset \Ker(f_{n-1}f_n) \subset \cdots\subset \Ker(f_1\cdots f_n)=M
\]
is a $\X$-filtration of $M$. Moreover, $\frac{\Ker(f_i\cdots f_n)}{\Ker(f_{i+1}\cdots f_n)}\cong \Ker f_i$ for all $i\in\{1, 2, \dots, n-1\}$.
\end{enumerate}
\end{proposition}

We are now in position to demonstrate our main result.

\begin{theorem}\label{Thm 1.5} Let $A$ be an artinian algebra and let $\C$ be a nonzero extension-closed subcategory of $\mod A$.
\begin{enumerate}
\item[(a)] If $\C$ is closed under kernels of epimorphisms, then $\CQ$ is the smallest class of modules such that $\F(\CQ)=\C$.
\item[(b)] If $\C$ is closed under cokernels of monomorphisms, then $\CS$ is the smallest class of modules such that $\F(\CS)=\C$. 
\end{enumerate}
\end{theorem}
\begin{proof}
(a). Let $M\in\F(\CQ)$. Then there exists a sequence of submodules  
\[
0=M_0\subset M_1 \subset \cdots \subset M_{n-1}\subset M_n=M
\]
of $M$ such that $M_i/M_{i-1}\in \CQ$ for all $i=1, 2, \dots, n$.  In particular, $M_1\in \CQ\subseteq \C$.  By induction, suppose that $M_{i-1}\in\C$ for some $i\geq 2$. The short exact sequence $0\rightarrow M_{i-1}\rightarrow M_i\rightarrow M_i/M_{i-1}\rightarrow 0$, together with the fact that $\C$ is closed under extensions gives $M_i\in\C$.    So $M\in\C$, and $\F({\CQ})\subseteq \C$.

Conversely, suppose that $M$ is a nonzero module in $\C$.  If $M\in\CQ$, then $M\in\F(\CQ)$. Else, it follows from (\ref{Lem 1.3})(b) that there exists a proper quotient $f_1:M\rightarrow M/M_1$, with $M/M_1\in \CQ$. Since $\C$ is closed under kernels of epimorphisms, we get $M_1\in \C$. If $M_1\in\CQ$, then $M\in\F(\CQ)$.  Else, by repeating this process, we obtain a sequence of proper nonzero submodules
\[
\cdots \subset M_3 \subset M_2 \subset M_1 \subset M
\]
that must stop since $A$ is artinian. So $\CQ$ filters $\C$.

To demonstrate the minimality of $\CQ$, suppose that $\mathcal{C}_0$ is a class of modules in $\C$ that filters $\C$.  Suppose that $X\in \CQ\setminus\mathcal{C}_0$.  Since $X\in\CQ\subseteq\F(\mathcal{C}_0)$, there exists a proper nonzero submodule $Y$ of $X$ such that $X/Y\in\mathcal{C}_0$.  But $\mathcal{C}_0\subseteq\C$ implies $X/Y\in\C$, a contradiction to $X\in\CQ$.  So $\CQ\subseteq\mathcal{C}_0$. 

(b). 
The dual argument, using (\ref{Lem 1.3})(a) instead of (\ref{Lem 1.3})(b), shows that $\CS$ is the smallest class of modules that cofilters $\C$. The conclusion then follows from (\ref{Prop 1.7}). 
\end{proof}

\subsection{Modules of finite projective dimension}\label{Modules of finite projective dimension}
A classical example of a subcategory of $\mod A$ that is closed under extensions, kernels of epimorphisms and cokernels of monomorphisms consists of 
\[
\P=\{M\in\mod A \ | \ \pd  M<\infty\},
\]
that is, the set of all finitely generated modules of finite projective dimension.  It then follows from (\ref{Rem 1.1(a)(b)}) that $\und{\P}=\ov{\P}$. Moreover, it follows from (\ref{Thm 1.5}) and (\ref{Prop 1.7}) that $\und{\P}=\ov{\P}$ is the smallest class of modules that filters $\P$ and the smallest class of modules that cofilters $\P$. 

Observe that if $A$ is of finite global dimension, then $\P=\mod A$, and thus $\und{\P}=\ov{\P}$ consists of all simple $A$-modules.  After this remark, one might expect that, given a subcategory $\C$ of $\mod A$, the classes ${\CS}$ and ${\CQ}$ share nice properties with the simple $A$-modules, for instance the Jordan-H\"older property or, more simply, its finiteness. Examples (b) and (c) below show that it is unfortunately not the case in general.

\begin{examples}\label{Ex 1}
In the following examples, the algebras $A$ are assumed to be $k$-algebras, where $k$ is an algebraically closed field.
\begin{enumerate}
\item[(a)] Let $A$ be the algebra given by the quiver 
$\SelectTips{eu}{10}\xymatrix@1{\bullet \ar@(ul, ur)[]^\alpha}$ bound by the relation $\alpha^2=0$. Then, up to isomorphisms, $A$ has only two indecomposable modules, namely a simple module $S$ and its minimal projective cover $P$.  Since the simple $S$ has infinite projective dimension, we have $\und{\P}=\{P\}$. 
\item[(b)] Let $A$ be the algebra given by the quiver 
\[
\SelectTips{eu}{10}\xymatrix{&&1\ar[dl]_-{\alpha}\ar[dr]^{\beta} \\
& 2 \ar[dl]_-{\gamma}\ar[dr]^{\delta} && 4 \ar[dl]^-\varepsilon \\
3\ar[dr]_-{\zeta} && 5\ar[dl]^-\eta \ar@(d, r)[]_-\theta \\
& 6 \ar@(dl, dr)[]_-\iota}
\] 
bound by the commutativity relations $\alpha\delta=\beta\varepsilon$ and $\gamma\zeta=\delta\eta$; and all compositions with loops: $\delta\theta=\varepsilon\theta=\theta^2=\theta\eta=\zeta\iota=\iota^2=\eta\iota=0$. One can verify that $\und{\P}$ contains seven elements, represented below by their Loewy series:
 \[
 \und{\P}=\left\{
\begin{array}{c}
	1\\4
\end{array}, 
\begin{array}{c}
	2\\5
\end{array},
\begin{array}{c}
	3\\6
\end{array},
\begin{array}{c}
	1\\2\\3
\end{array},
\begin{array}{c}
	4\\5\\6
\end{array},
\begin{array}{c}
	5\\5 \ 6
\end{array},
\begin{array}{c}
	6\\6
\end{array}
\right\}.
 \]
Observe that $\und{\P}$ contains more elements than the number of non-isomorphic simple $A$-modules, Moreover, the indecomposable projective module associated with the vertex $1$, namely 
$P_1=\begin{array}{c}
	\ \ \ 1 \\ 
	\ \ \; 2 \ \ 4 \\
	3 \ \ 5 \\
	\ 6
\end{array}$, admits the following $\und{\P}$-filtrations:
\[
0\subset \begin{array}{c}
	4 \\ 
	5 \\
	6
\end{array}\subset \begin{array}{c}
	\ \ \ 1 \\ 
	\ \ \; 2 \ \ 4 \\
	3 \ \ 5 \\
	\ 6
\end{array}
\qquad\text{and}\qquad
0\subset \begin{array}{c}
	3\\ 
	6
\end{array}\subset \begin{array}{c}
	\ \ 2 \\ 
	\ 3 \ \ \;5 \\
	\ \ 6 
\end{array}\subset \begin{array}{c}
	\ \ \ 1 \\ 
	\ \ \; 2 \ \ 4 \\
	3 \ \ 5 \\
	\ 6
\end{array}.
\]
It is worthwhile to observe here that, in addition to involve different elements of $\und{\P}$, both composition series have different length, by opposition to Jordan-H\"older's theorem for the simple modules.
\item[(c)] This example, due to Igusa, Smal{\o} and Todorov (see \cite{IST90}), shows that $\und{\P}$ can contain infinitely many elements.

Let $A$ be the algebra given by the quiver 
$\SelectTips{eu}{10}\xymatrix{1 \ar@/_1pc/[r]_-\gamma & 2 \ar@/_1pc/[l]_-\alpha \ar[l]_-\beta}$ bound by the relations $\gamma\alpha=\gamma\beta=\alpha\gamma=0$. Then both simple modules $S_1$ and $S_2$ have infinite projective dimensions, and there is an infinite family of indecomposable $A$-modules $\{Y_i\}$, of length two, of projective
dimension one, annihilated by $\gamma$, and indexed over the projective line of
the field $k$ (excluding one point) with a nonzero $A$-morphism to $S_2$. It is shown in \cite{IST90} that these maps do not factor through an $A$-module of finite projective dimension, proving that $\P$ is not contravariantly finite.  Since the only proper submodule of each of these modules $Y_i$ is the simple module $S_1$, these modules $Y_i$ are in $\und{\P}$, proving at the same time that  $\und{\P}$ is infinite and that $\und{\P}$ is not contravariantly finite in $\mod A$.

We refer the reader to \cite[Chapter 4]{Sto04} for a more in-depth investigation of the modules in $\und{\P}$ for this specific algebra.

\item[(d)] Let $A$ be the algebra given by the quiver 
$\SelectTips{eu}{10}\xymatrix{1\ar[r]^\alpha & 3 \ar@(ul,ur)[]^\beta & 2 \ar[l]_{\gamma}}$
bound by $\alpha\beta=0$ and $\beta^2=0$. The Auslander-Reiten quiver of $\mod A$ is given in the Figure 1 below,
\begin{figure}[!htn]
      \resizebox{.8\linewidth}{!}{\includegraphics{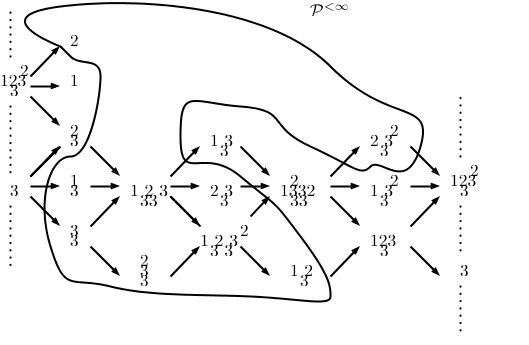}}
      \caption{}
\end{figure}
in which we depicted the modules in $\P$. It is then easy to verify that the minimal right $\P$-approximations of the simple modules are given by 
$\X=\left\{
\begin{array}{c}
	1 \ 2 \\3
\end{array},
2, 
\begin{array}{c}
	2 \ 3 \\3
\end{array}
\right\}$, 
while $\ov{\P}=\und{\P}=\left\{
\begin{array}{c}
	1 \\3
\end{array},
2,
\begin{array}{c}
	3\\3
\end{array}
\right\}$.
By inspection, it is easy to verify that every module in $\P$ is filtered by $\und{\P}$.  For instance, a $\und{\P}$-filtration of 
$M=\begin{array}{c}
	1 \ 2 \ 3\\3 \ 3
\end{array}$ is given by
\[
0\subset \begin{array}{c}
	1 \\ 
	3 
\end{array}\subset \begin{array}{c}
	1 \\
	3
\end{array}\bigoplus
\begin{array}{c}
	3 \\
	3
\end{array}
\subset
\begin{array}{c}
	1 \ 2 \ 3\\3 \ 3
\end{array},
\]
while
\[
\begin{array}{c}
	1 \ 2 \ 3\\3 \ 3
\end{array}
\longrightarrow 
\begin{array}{c}
	2 \ 3 \\
	3
\end{array}\longrightarrow
2\longrightarrow 0
\]
is a $\und{\P}$-cofiltration of $M$.

On the other hand, the result of Auslander and Reiten mentioned in the introduction asserts that $M$ is a direct summand of a $\X$-filtered module $\widetilde{M}$. Here, one has to take $\widetilde{M}=\begin{array}{c}
	1 \ 2 \ 3\\3 \ 3
\end{array}\bigoplus 
\begin{array}{c}
	 \  \ \ 2\\2 \ 3\\3
\end{array}$ and the (less intuitive) $\X$-filtration
\[
0\subset \begin{array}{c}
	2 \ 3 \\ 
	3 
\end{array}\subset \begin{array}{c}
	3 \\
	3
\end{array}\bigoplus
\begin{array}{c}
	 \  \ \ 2\\2 \ 3\\3
\end{array}
\subset
\begin{array}{c}
	1 \ 2 \ 3\\3 \ 3
\end{array}\bigoplus 
\begin{array}{c}
	 \  \ \ 2\\2 \ 3\\3
\end{array}
\]
and $ \X$-cofiltration 
\[
\begin{array}{c}
	1 \ 2 \ 3\\3 \ 3
\end{array}\bigoplus 
\begin{array}{c}
	 \  \ \ 2\\2 \ 3\\3
\end{array}
\longrightarrow
\begin{array}{c}
	1 \ 2 \\3 
\end{array}\bigoplus 
\begin{array}{c}
2 \ 3\\3
\end{array}
\longrightarrow
\begin{array}{c}
1 \ 2\\3
\end{array}
\longrightarrow 0.
\]
\end{enumerate}
\end{examples}

These examples show that $\und{\P}$ is sometimes difficult to track.  This set is however of great interest, as demonstrated by the two next results in connection with the finitistic dimension.  Recall that the \textbf{finitistic dimension} of an algebra $A$, denoted by $\findim A$,  is by definition the supremum of the projective dimensions of the modules in $\P$. It is still an open question whether $\findim A$ is always finite.

\begin{lemma}\label{Cor findim}
If $A$ is an artinian algebra, then $\findim A=\sup\{\pd  X \ | \ X\in\und{\P}\}$.
\end{lemma}
\begin{proof}
Clearly, $\findim A\geq\sup\{\pd  X \ | \ X\in\und{\P}\}$.  On the other hand, if $M\in\P$, then there exists a $\und{\P}$-filtration
\[
0=M_0\subset M_1\subset M_2\subset \cdots\subset M_n=M.
\]
Then it follows by induction that 
\[
\pd  M\leq \sup\{\pd  M_0, \pd  M_1/M_0, \dots, \pd  M_n/M_{n-1}\}\leq \sup\{\pd  X \ | \ X\in\und{\P}\},
\]
so that $\findim A\leq\sup\{\pd  X \ | \ X\in\und{\P}\}$.
\end{proof}

As mentioned in the introduction of the paper, the idea of the following proof is due to D. Happel; compare with \cite[Section 2.3]{Hap}. 

\begin{theorem}\label{Thm Happel}
Let $A$ be a finite-dimensional algebra over a field $k$. If the length of every indecomposable module in $\und{\P}$ is bounded, then $\findim A<\infty$.
\end{theorem}
\begin{proof}
By (\ref{Cor findim}), it is enough to show that the projective dimension of modules in $\und{\P}$ is bounded.  For a given integer $r$ we consider the algebraic variety $\mathcal{M}_r(A)$ of $A$-modules of length $r$.  Consider the subset $\mathcal{P}_t$ of $\mathcal{M}_r(A)$ formed by those modules $X$ which satisfy $\pd  X \leq t$.  By \cite[(Theorem 12.61)]{JL89}, it follows that $\mathcal{P}_t$ is an open subset of $\mathcal{M}_r(A)$ in the Zariski topology.  So we obtain an ascending chain of open subsets $\mathcal{P}_0\subseteq \mathcal{P}_1 \subseteq \cdots$ which has to become stationary since $\mathcal{M}_r(A)$ is finite-dimensional.$\mathcal{P}_t$ is an open subset of $\mathcal{M}_r(A)$  So there exists $t_r$ with $\pd X\leq t_r$ for all $X\in \mathcal{M}_r(A)$.  

By assumption, there is an integer $m$ such that the length of indecomposable modules in $\und{\P}$ is bounded by $m$.  Hence $\findim A=\sup\{\pd  X \ | \ X\in\und{\P}\}\leq \max\{t_r \ | \ 1\leq r\leq m\}$, which shows the assertion.
\end{proof}

\begin{remark}\label{Prop 3.13}
We observed above that the Jordan-H\"older property is not satisfied in general for $\F({\CS})$ and $\F({\CQ})$.  However, it is an easy exercise to verify that if ${\CS}={\CQ}$ and this set consists only of simple $A$-modules, then:
\begin{enumerate}
\item[(a)] The classical proof of the Jordan-H\"older Theorem holds true on $\F({\CS})=\F({\CQ})$. 
\item[(b)] If \[
0=M_0\subset M_1 \subset \cdots \subset M_{n-1}\subset M_n=M
\]
is a ${\CQ}$-filtration of $M$ and $L$ is a submodule of $M$, then 
\begin{enumerate}
\item[(i)] $0=L\cap M_0\subset L\cap M_1 \subset \cdots \subset L\cap M_{n-1}\subset L\cap M_n=L$ is such that every consecutive quotient is zero or in ${\CQ}$.
\item[(ii)] $0=\frac{L+M_0}{L}\subset \frac{L+M_1}{L} \subset \cdots \subset \frac{L+M_{n-1}}{L}\subset \frac{L+M_n}{L}=\frac{M}{L}$ is such that every consecutive quotient is zero or in ${\CQ}$.
\end{enumerate}
\end{enumerate}
\end{remark}

Observe that the hypotheses of the above remark hold true when, for instance, $\C$ is closed under submodules and quotients, see (\ref{Lem 3.11}).
The following proposition is useful to determine when an extension-closed subcategory is closed under submodules and quotients.

\begin{proposition}\label{Prop 3.10}
Let $\C$ be a nonzero subcategory of $\mod A$ that contains the zero module and is closed under extensions.
\begin{enumerate}
\item[(a)] If $\C$ is closed under submodules, then $\C$ is closed under quotients if and only if $\CQ\subseteq \CS$. Moreover, in this case, we have $\CQ=\CS$.
\item[(b)] If $\C$ is closed under quotients, then $\C$ is closed under submodules if and only if $\CS\subseteq \CQ$. Moreover, in this case, we have $\CQ=\CS$.
\end{enumerate}
\end{proposition}
\begin{proof}
We only prove (a) since the proof of (b) is dual. The necessity directly follows from (\ref{Lem 3.11}). For the sufficiency, if $\CQ\subseteq \CS$, then $\CQ=\CS$ by (\ref{Rem 1.1(a)(b)}). Then (\ref{Lem 3.11}) implies that $\CQ=\CS=\{S \ | \ S \text{ is a simple module in } \C\}$ and $\C=\F(\{S \ | \ S \text{ is a simple module in } \C\})$ since $\C=\F(\CQ)$ by (\ref{Thm 1.5}). Now, let $M\in\C$ and $f:M\rightarrow N$ be an epimorphism.  If $N=0$, then $N\in\C$ by assumption. Otherwise, every simple $A$-module that appears in any composition series of $N$ also appears in any composition series of $M$, showing that $N\in \F(\{S \ | \ S \text{ is a simple module in } \C\}\subseteq\C$.
\end{proof}

\subsection{Auslander-Reiten Quiver}

Let $A$ be an artin algebra and $\C$ be a subcategory of $\mod A$. In this section, we gather some observations about the positioning of the elements of ${\CS}$ and ${\CQ}$ in the Auslander-Reiten quiver of $\mod A$.  Recall from (\ref{Rem 1.1 (c)}) that the elements of ${\CS}$ and ${\CQ}$ often consist of indecomposable $A$-modules, and so can be seen as vertices in the Auslander-Reiten quiver of $\mod A$. For further details on the Auslander-Reiten quiver and on its connected components, we refer the reader to \cite{ASS06, SS07}, for instance.

\begin{proposition}\label{Prop ASS}
Let $\C$ be a nonzero subcategory of $\mod A$ that is closed under extensions and direct summands. 
Suppose that
\[
\SelectTips{eu}{10}\xymatrix{0\ar[r]& L\ar[r]^-{f=(f_i)} & M=\displaystyle\bigoplus_{i=1}^n M_i \ar[r]^-{g=(g_i)} & N \ar[r] & 0}
\]
is an almost split sequence in $\mod A$, with $L\in{\CQ}$ and $N\in{\CS}$. Then $M$ is indecomposable and $M\notin{\CS}\cup{\CQ}$.
\end{proposition}
\begin{proof}
Since $\C$ is closed under extensions and direct summands, we have $M\in\C$ and $M_i\in\C$ for all $i=1, 2, \dots, n$.  Now each $f_i$ is an irreducible morphisms.  In particular, each $f_i$ is a nonisomorphism that is either an epimorphism or a monomorphism. But because $L\in{\CQ}$, the $f_i's$ cannot be epimorphisms, and thus are (strict) monomorphisms. Dually, because $N\in{\CS}$, each $g_i$ is a (strict) epimorphism.  So 
$\ell(L)<\ell(M_i)$ and $\ell(N)<\ell(M_i)$ for all $i$.  Consequently, if $n\geq 2$, we get
$\ell(L)+\ell(N) < \sum_{i=1}^n\ell(M_i)$, a contradiction.  Thus $n=1$ and $M$ is indecomposable.  In addition, $M\notin{\CS}\cup{\CQ}$ since $f$ is a monomorphism and $g$ an epimorphism.
\end{proof}

\begin{proposition}\label{Prop connected components}
Let $\C$ be a nonzero subcategory of $\mod A$, and let $\Gamma$ be a connected component of the Auslander-Reiten quiver of $\mod A$.  
\begin{enumerate}
\item[(a)] If $\Gamma$ is postprojective and $\Gamma\subseteq \C$, then every module in $\Gamma\cap{\CS}$ is simple.
\item[(b)] If $\Gamma$ is preinjective and $\Gamma\subseteq \C$, then every module in $\Gamma\cap{\CQ}$ is simple.
\item[(c)] If $\Gamma$ is a regular tube of rank $n$, then it contains at most $n$ elements in ${\CS}$, and at most $n$ elements in ${\CQ}$.
\end{enumerate}
\end{proposition}
\begin{proof}
(a). Suppose that $M\in \Gamma\cap{\CS}$.  If $M$ is not simple, it has a simple submodule $S$.  Since $\Gamma$ is a postprojective component, it is closed under submodules, so $S\in \Gamma\subseteq \C$, a contradiction to $M\in{\CS}$.  So $M$ is simple. The proof of (b) is dual.

(c). This follows from the fact that given any $n+1$ modules $M_1, M_2, \dots, M_{n+1}$ in a regular tube of rank $n$, there are at least two modules $M_i$ and $M_j$ in this list lying in the ray starting at a quasi-simple module $M$.  Thus there is a monomorphism $M_i\longrightarrow M_j$ or $M_j\longrightarrow M_i$, which is a contradiction if these two modules are in ${\CS}$.  Dually, there are at least two modules $M_k$ and $M_l$ in this list lying in the coray ending at a quasi-simple module $N$.  Thus there is an epimorphism $M_k\longrightarrow M_l$ or $M_l\longrightarrow M_k$, which is a contradiction if these two modules are in ${\CQ}$.
\end{proof}

\section{Inclusions of subcategories}

The purpose of this section is to study the behavior of specific inclusions of subcategories of $\mod A$ in order to obtain, in particular, equivalent conditions to the finiteness of the finitistic dimension of $A$. As in Section 1, $A$ is an artinian algebra.

Suppose that we have a sequence of inclusions of subcategories
\[
\{0\}\neq \C_0\subseteq \C_1\subseteq \C_2 \subseteq \cdots\subseteq  \mod A.
\]
We will consider the following sets of properties: if 
\[
\SelectTips{eu}{10}\xymatrix{0\ar[r]&L\ar[r]&M\ar[r]&N\ar[r] &0}
\]
 is a short exact sequence in $\mod A$, then
\begin{enumerate}
\item[(G1)] 
\begin{enumerate}
\item[(a)]  $L\in\C_l$ and $M\in \C_m$ implies $N\in \C_{\max\{l+1, m\}}$
\item[(b)]  $M\in\C_m$ and $N\in \C_n$ implies $L\in \C_{\max\{m, n-1\}}$
\end{enumerate}
\item[(G2)]
\begin{enumerate}
\item[(a)]  $L\in\C_l$ and $M\in \C_m$ implies $N\in \C_{\max\{l-1, m\}}$
\item[(b)]  $M\in\C_m$ and $N\in \C_n$ implies $L\in \C_{\max\{m, n+1\}}$
\end{enumerate}
\item[(G3)] $L\in\C_l$ and $N\in \C_n$ implies $M\in \C_{\max\{l, n\}}$
\end{enumerate}

\begin{lemma}\label{Dans la preuve du Cor 2.3} Let $\{0\}\neq \C_0\subseteq \C_1\subseteq \C_2 \subseteq \cdots\subseteq  \mod A$ be inclusions of subcategories of $\mod A$.
\begin{enumerate}
\item[(a)] If (G1) holds true, then $\C_n$ is closed under kernels of epimorphisms for all $n$. In particular, $\und{\C_n}\subseteq \ov{\C_n}$ for all $n\geq 0$.
\item[(b)] If (G2) holds true, then $\C_n$ is closed under cokernels of monomorphisms for all $n$. In particular, $\ov{\C_n}\subseteq \und{\C_n}$ for all $n\geq 0$.
\end{enumerate}
\end{lemma}
\begin{proof}
We only prove (a) since the proof of (b) is dual. Let $f:M\rightarrow N$ be an epimorphism, with $M,N\in\C_n$.  By (G1)(b), we get $\Ker f\in \C_n$. The second part follows from (\ref{Rem 1.1(a)(b)}).
\end{proof}

\begin{lemma}\label{Thm 2.5} Let $\{0\}\neq \C_0\subseteq \C_1\subseteq \C_2 \subseteq \cdots\subseteq  \mod A$ be inclusions of subcategories of $\mod A$. Suppose that (G3) holds true, and that (G1) or (G2) holds true. Let $m,n\in \{1, 2, \dots\}$ with $m\geq n$. The following conditions are equivalent:
\begin{enumerate}
\item[(a)] $\C_n=\C_m$,
\item[(b)] $\ov{\C_m}\subseteq \C_n$,
\item[(c)] $\und{\C_m}\subseteq \C_n$.
\end{enumerate}
\end{lemma}
\begin{proof}
We suppose that (G1) holds true.  The proof when (G2) holds true is similar. It is clear that (a) implies (b), while (b) implies (c) by (\ref{Dans la preuve du Cor 2.3}). To complete the proof, it suffices to show that $\C_m\subseteq \C_n$ when $\und{\C_m}\subseteq \C_n$. Suppose that $M\in\C_m\setminus\C_n$.  By (\ref{Lem 1.3})(a), there exists a (nonzero) submodule $L_0$ of $M$ such that $L_0\in\und{\C_m}$.  Since $\und{\C_m}\subseteq \C_n$ by assumption, we get $L_0\neq M$. Moreover, because $m> n$, it follows from (G1)(a) that $M/L_0\in \C_m$. Set $M_1=M/L_0$.  Observe that $M_1\notin \C_n$, since otherwise $M\in\C_n$ by (G3), a contradiction. So $M_1\in\C_m\setminus \C_n$. Repeating this process gives an infinite sequence of non-invertible epimorphisms
\[
\SelectTips{eu}{10}\xymatrix{M\ar[r]^{f_1}&M_{1} \ar[r]^{f_2} & M_2  \ar[r]^{f_3}& \cdots},
\]
that is, an infinite sequence of strict inclusions 
\[
\Ker(f_1) \subset \Ker(f_{2}f_1) \subset \Ker(f_3f_{2}f_1) \subset \cdots\subset M,
\]
which contradicts the fact that $A$ is noetherian.  Therefore $\C_m\subseteq \C_n$.
\end{proof}

\begin{theorem}\label{Thm 2.6} Let $\{0\}\neq \C_0\subseteq \C_1\subseteq \C_2 \subseteq \cdots\subseteq  \mod A$ be inclusions of subcategories of $\mod A$ satisfying (G3). Let $n\geq 0$. Moreover, suppose that at least one of the following two properties holds true:
\begin{enumerate}
\item[(i)] (G1) holds true, and for all $M\in \C$, there exists an epimorphism $\pi_M:P_M\rightarrow M$ with $P_M\in \C_0$. 
\item[(ii)] (G2) holds true, and for all $M\in \C$, there exists a monomorphism $\sigma_M:M\rightarrow I_M$ with $I_M\in \C_0$.
\end{enumerate}
Then the following conditions are equivalent:
\begin{enumerate}
\item[(a)] $\C_n=\C_m$ for all $m\geq n$,
\item[(b)] $\ov{\C_{n+1}}\subseteq \C_n$,
\item[(c)] $\und{\C_{n+1}}\subseteq \C_n$.
\end{enumerate}
\end{theorem}
\begin{proof}
We only prove the assertion when (i) holds true, the proof when (ii) holds true is dual. 
It is clear that (a) implies (b), while (b) implies (c) by (\ref{Dans la preuve du Cor 2.3}). To show that (c) implies (a), suppose that $\und{\C_{n+1}}\subseteq \C_n$ and let $M_0\in\C_m$, with $m>n$. By assumption, there exists an epimorphism $f_0:P_{M_0}\rightarrow M_0$, with $P_{M_0}\in\C_0$. Then it follows from (G1)(b) that $\Ker f_0\in \C_{m-1}$.  Set $M_1=\Ker f_0$.  By repeating this process, we obtain a family of short exact sequences $0\rightarrow M_{i+1}\rightarrow P_{M_i}\rightarrow M_i\rightarrow 0$, where $M_i\in\C_{m-i}$ and $P_{M_i}\in\C_0$ for all $i=0, 1, \dots, m$. In particular, $M_{m-n-1}\in\C_{n+1}$. By hypothesis, we have $\und{\C_{n+1}}\subseteq \C_n$. By (\ref{Thm 2.5}), we get $\C_{n+1}=\C_n$ and $M_{m-n-1}\in\C_n$. Then, the short exact sequence $0\rightarrow M_{m-n-1}\rightarrow P_{M_{m-n-2}}\rightarrow M_{m-n-2}\rightarrow 0$ together with (G1)(a) gives $M_{m-n-2}\in\C_{n+1}=\C_n$.  By induction, we get $M_i\in\C_n$ for all $i=0, 1, \dots, m-n-1$.  In particular, $M_0\in\C_n$.  So $\C_m=\C_n$ for all $m\geq n$.
\end{proof}

\begin{conjecture}\label{Conj} 
In the setting of (\ref{Thm 2.6}), we conjecture that $\C_n=\C_m$ for all $m\geq n$ if and only if $\ov{\C_n}=\und{\C_n}$.
\end{conjecture}

\begin{example}\label{ex findim}
Let $\mathcal{C}$ be a subcategory of $\mod A$. For each $n\geq0$, let
\[
\C_n= \ ^{\perp_n}\mathcal{C}=\{M\in\mod A \ | \ \Ext^i_A(M, -)\mid_{\mathcal{C}}=0, \text{ for all }i>n\}.
\]
We then have inclusions of subcategories $\{0\}\neq\C_0\subseteq \C_1\subseteq \C_2 \subseteq \cdots \subseteq \mod A$ satisfying the conditions (G1) and (G3). Moreover, since all projective $A$-modules are in $\C_0$, the hypothesis (i) of (\ref{Thm 2.6}) holds true.  We can therefore apply (\ref{Thm 2.5}) and (\ref{Thm 2.6}).

In particular, if we take $\mathcal{C}=\mod A$, then 
\[
\C_n=\{M\in\mod A \ | \ \pd  M\leq n\},
\] that is the set of all finitely generated $A$-modules of projective dimension at most $n$.  This set is often denoted by $\mathcal{P}^{\leq n}$ in the literature.

In this case, (\ref{Thm 2.6}) states that the following conditions are equivalent:
\begin{enumerate}
\item[(a)] $\findim A\leq n$.
\item[(b)] $\und{\mathcal{P}^{\leq n+1}}\subseteq \mathcal{P}^{\leq n}$.
\item[(c)] $\ov{\mathcal{P}^{\leq n+1}}\subseteq \mathcal{P}^{\leq n}$.
\end{enumerate}
Observe that $\und{\mathcal{P}^{\leq n+1}}\subseteq \ov{\mathcal{P}^{\leq n+1}}$ by (\ref{Dans la preuve du Cor 2.3}), so verifying $\und{\mathcal{P}^{\leq n+1}}\subseteq \mathcal{P}^{\leq n}$ is, at least in theory, easier than verifying $\ov{\mathcal{P}^{\leq n+1}}\subseteq \mathcal{P}^{\leq n}$.

Dually, given a subcategory $\mathcal{C}$ of $\mod A$, one can consider, for each $n\geq0$, the set
\[
\C_n= \mathcal{C}^{\perp_n}=\{M\in\mod A \ | \ \Ext^i_A(-, M)\mid_{\mathcal{C}}=0, \text{ for all }i>n\}.
\]
We then have inclusions of subcategories $\{0\}\neq\C_0\subseteq \C_1\subseteq \C_2 \subseteq \cdots \subseteq \mod A$ satisfying the conditions (G2) and (G3). Moreover, since all injective $A$-modules are in $\C_0$, the hypothesis (ii) of (\ref{Thm 2.6}) holds true.  We can therefore apply (\ref{Thm 2.5}) and (\ref{Thm 2.6}).

In particular, if we take $\mathcal{C}=\mod A$, then 
\[
\C_n=\{M\in\mod A \ | \ \id_A M\leq n\},
\] that is the set of all finitely generated $A$-modules of injective dimension at most $n$. 
\end{example}

We end this section by showing that ${\CS}={\CQ}$ when (G1) or (G2) holds true, generalizing this observation in the particular case $\P=\bigcup_{n\geq 0}\mathcal{P}^{\leq n}$, as observed at the beginning of Section \ref{Modules of finite projective dimension}.

\begin{proposition}\label{Lem 2.4} Let $\{0\}\neq \C_0\subseteq \C_1\subseteq \C_2 \subseteq \cdots\subseteq  \mod A$ be inclusions of subcategories of $\mod A$. Let $\C=\bigcup_{n\geq 0}\C_n$. If (G1) or (G2) holds true, then: 
\begin{enumerate}
\item[(a)] $\displaystyle\bigcap_{i\geq n}\und{\C_i}=\bigcap_{i\geq n}\ov{\C_i}$ for all $n\geq 0$.
\item[(b)] $\displaystyle{\CS}=\bigcup_{n\geq 0}\left(\bigcap_{i\geq n}\und{\C_i}\right)=\bigcap_{n\geq 0}\left(\bigcup_{i\geq n}\und{\C_i}\right)={\CQ}$.
\end{enumerate}
\end{proposition}
\begin{proof}
We suppose that (G1) holds true.  The proof when (G2) holds true is similar. Let $n\geq 0$. First, it is easy to see that 
\[
\bigcap_{i\geq n}\und{\C_i}=\{M\in\C_n \ | \ M\neq 0 \text{ and } 0\neq L \subset M \text{ implies } L\notin \C\}
\]
and
\[
\bigcap_{i\geq n}\ov{\C_i}=\{M\in\C_n \ | \ M\neq 0 \text{ and } 0\neq L \subset M \text{ implies } M/L\notin \C\}
\]

(a).  Suppose that $M\in \bigcap_{i\geq n}\und{\C_i}$. Consequently, $M\in\C_n$ and if $L$ is a nonzero proper submodule of $M$, then $L\notin \C$. Consequently, it follows from (G1)(b) that $M/L\notin \C$.  This shows that $M\in \bigcap_{i\geq n}\ov{\C_i}$, and thus $\bigcap_{i\geq n}\und{\C_i}\subseteq\bigcap_{i\geq n}\ov{\C_i}$.   Similarly, we show the reverse inclusion by using (G1)(a).

(b). We have
\[
\begin{array}{rcl}
{\CS}&=&\{M\in\C \ | \ M\neq 0 \text{ and } 0\neq L \subset M \text{ implies } L\notin \C\} \\
&=&\{M\in\bigcup_{n\geq 0}\C_n \ | \ M\neq 0 \text{ and } 0\neq L \subset M \text{ implies } L\notin \C\} \\
& = & \bigcup_{n\geq 0}\left(\bigcap_{i\geq n}\und{\C_i}\right)\\
& = & \bigcup_{n\geq 0}\left(\bigcap_{i\geq n}\ov{\C_i}\right) \text{(by (a))} \\
&=&\{M\in\bigcup_{n\geq 0}\C_n \ | \ M\neq 0 \text{ and } 0\neq L \subset M \text{ implies } M/L\notin \C\} \\
&=&\{M\in\C \ | \ M\neq 0 \text{ and } 0\neq L \subset M \text{ implies } M/L\notin \C\} \\
&=&{\CQ}
\end{array}
\]
\end{proof}

\section{Torsion pairs and modules with (co)standard filtrations}

The torsion pairs provide a rich family of extension-closed subcategories on which our results apply.  Indeed, it is well known that if $(\T, \F)$ is a torsion pair, then $\F$ is closed under extensions and kernels of epimorphisms, while $\T$ is closed under extensions and cokernels of monomorphisms (see Proposition \ref{Prop Torsion}).  The torsion pairs play an important role in representation theory of algebras, specifically related to tilting theory, and more generally to $\tau$-tilting theory, see \cite{AIR14}. In this section, we will have a closer look at the classes $\ov{\F}$ and $\und{\T}$ and their interactions. In particular, we show that these classes are totally compatible with the Brenner-Butler Theorem when $(\T, \F)$ is induced by a tilting module (Theorem \ref{Prop 3.8}).

The second part of this section concerns the modules with standard or costandard filtrations, in connection with standardly stratified and quasi-hereditary algebras.

In the sequel, $A$ stands for an artin algebra, thus, in particular, an artinian algebra.

\subsection{Torsion pairs}

We first recall that a pair $(\mathcal{T}, \mathcal{F})$ of full subcategories of $\mod A$ is called a \textbf{torsion pair} if, for all $M\in\mod A$, we have 
\begin{enumerate}
\item[(a)] $\Hom_A(M,-)\mid_{\F}=0$ if and only if $M\in\T$, and
\item[(b)] $\Hom_A(-,M)\mid_{\T}=0$ if and only if $M\in\F$.
\end{enumerate}
In this case, $\T$ is called the \textbf{torsion class} and $\F$ is called the \textbf{torsion-free class}.

The following proposition gathers several known results about torsion pairs, see for instance \cite[(Chapter VI)]{ASS06}.

\begin{proposition}\label{Prop Torsion}
Let $(\T, \F)$ be a torsion pair.
\begin{enumerate}
\item[(a)] $\T$ is closed under quotients (in particular, cokernels of monomorphisms), direct sums, and extensions.
\item[(b)] $\F$ is closed under submodules (in particular, kernels of epimorphisms), direct products, and extensions.
\item[(c)] There exists a subfunctor $t$ of the identity functor on $\mod A$ such that, for all $M\in \mod A$, we have:
\begin{enumerate}
\item[(i)] $tM\in \T$,
\item[(ii)] $M/tM\in\F$,
\item[(iii)] $tM=M$ if and only if $M\in\T$,
\item[(iv)] $tM=0$ if and only if $M\in \F$,
\end{enumerate}
\end{enumerate}
\end{proposition}

It thus follows from (\ref{Thm 1.5}) and (\ref{Prop 1.7}) that $\und{\T}$ is the smallest class of modules that filters (and cofilters) $\T$, while $\ov{\F}$ is the smallest class of modules that filters (and cofilters) $\F$.  Moreover, we get the following nice observations.

\begin{lemma}\label{Prop 3.2 and Cor 3.14}
Let $(\T, \F)$ be a torsion pair, and let $M$ be a nonzero module in $\mod A$. 
\begin{enumerate}
\item[(a)] $M\in \T$ if and only if $\Hom_A(M,-)\mid_{\ov{\F}}=0$.
\item[(b)] $M\in \F$ if and only if $\Hom_A(-, M)\mid_{\und{\T}}=0$.
\item[(c)] If $M\in\T$, then $M\in \und{\T}$ if and only if every proper nonzero submodule of $M$ is in $\F$.
\item[(d)] If $M\in\F$, then $M\in \ov{\F}$ if and only if every proper nonzero quotient of $M$ is in $\T$.
\end{enumerate}
\end{lemma}
\begin{proof}
(a) and (b) follow from (\ref{Prop 1.9}) and the definition of a torsion pair. We will prove (c), the proof of (d) is dual.

(c). The sufficiency is clear since $\F\cap\T=\{0\}$. For the necessity, suppose that $M\in\und{\T}$. Suppose that $L$ is a proper nonzero submodule of $M$.  Then $tL\subseteq L\subset M$, which forces $tL=0$ since $M\in\und{\T}$ and $tL\in\T$. Consequently, $L\in\F$ by (\ref{Prop Torsion}).
\end{proof}

Before studying in greater details the torsion pairs induced by tilting modules, recall that a torsion pair $(\T, \F)$ is said to be \textbf{hereditary} if $\T$ is closed under submodules, and \textbf{cohereditary} if $\F$ is closed under quotients.

The (co)hereditary torsion pairs have been heavily studied and have connections with many other concepts, for instance the TTF-triples, the strongly complete Serre classes and the strongly complete filters of ideals of a ring.


It turns out that we can describe the hereditary and cohereditary torsion pairs as follows.

\begin{proposition}\label{Cor 3.12}
Let $(\T, \F)$ be a torsion pair in $\mod A$. If $\T$ and $\F$ are nonzero subcategories of $\mod A$, then
\begin{enumerate}
\item[(a)] $(\T, \F)$ is a hereditary torsion pair if and only if $\und{\T}\subseteq \ov{\T}$.
\item[(b)] $(\T, \F)$ is a cohereditary torsion pair if and only if $\ov{\F}\subseteq \und{\F}$.
\end{enumerate}
\end{proposition} 
\begin{proof}
This follows directly from (\ref{Prop 3.10}) and (\ref{Prop Torsion}). 
\end{proof}

Observe that if $(\T, \F)$ is a hereditary torsion pair, then it follows from (\ref{Rem 1.1(a)(b)}) and (\ref{Lem 3.11}) that $\und{\T}=\ov{\T}$ and that this set consists of all simple $A$-modules in $\T$.  In this case, one can apply the conclusions of (\ref{Prop 3.13}) on $\T$. Similarly, if $(\T, \F)$ is a cohereditary torsion pair, then $\und{\F}=\ov{\F}$ and this set consists of all simple $A$-modules in $\F$.

\subsection{Torsion pairs induced by tilting modules}

The purpose of this section is to show that $\und{\T}$ and $\ov{\F}$ are compatible with the Brenner-Butler Theorem when $(\T, \F)$ is a torsion pair induced by a tilting module. We first quickly recall the necessary background, see for instance \cite[(Chapter VI)]{ASS06} for more details.

Recall that an $A$-module $T$ is said to be a \textbf{tilting module} if 
\begin{enumerate}
\item[(a)] $\pd  T\leq 1$,
\item[(b)] $\Ext_A^1(T,T)=0$,
\item[(c)] There exists a short exact sequence $0\rightarrow A\rightarrow T'\rightarrow T''\rightarrow 0$, where $T'$ and $T''$ are direct sums of direct summands of the module $T$.
\end{enumerate}
Given a (right) tilting $A$-module $T$, let
\[
\T=\T(T)=\{M\in\mod A \ | \ \Ext_A^1(T,M)=0\}
\]
 and 
\[
\F=\F(T)=\{M\in\mod A \ | \ \Hom_A(T,M)=0\}.
\] 
Moreover, let $B=\End_AT$. Then $T$ is a (left) $B$-module. Consider 
\[
\X=\X(T)=\{X \in \mod B \ | \ X\otimes_BT=0\}
\] 
and 
\[
\Y=\Y(T)=\{ Y\in \mod B \ | \ \Tor_1^B(Y,T)=0\}.
\]
It is well known that $(\T, \F)$ is a torsion pair in $\mod A$ and $(\X, \Y)$ is a torsion pair in $\mod B$.

Under these assumptions and notations, the Brenner-Butler Theorem \cite{BB80} asserts that:
\begin{itemize}
\item The functors $\Hom_A(T,-):\mod A \rightarrow \mod B$ and $-\otimes_B T:\mod B\rightarrow \mod A$ induce quasi-inverse equivalence between $\T$ and $\Y$, and
\item The functors $\Ext^1_A(T,-):\mod A \rightarrow \mod B$ and $\Tor^B_1(-, T):\mod B\rightarrow \mod A$ induce quasi-inverse equivalence between $\F$ and $\X$.
\end{itemize}

The following remark gathers some known properties of the involved functors.  Their easy verifications are left to the reader.

\begin{remark} \label{Rem 3.4}
Let $T$ be a tilting $A$-module.
\begin{enumerate}
\item[(a)] The functor $\Hom_A(T, -)$ is left exact.  Moreover, if $0\rightarrow L\rightarrow M\rightarrow N\rightarrow 0$ is a short exact sequence in $\mod A$, with $L,M,N\in \T$, then the induced sequence $0\rightarrow \Hom_A(T,L)\rightarrow \Hom_A(T,M)\rightarrow \Hom_A(T,N)\rightarrow 0$ is exact in $\mod B$.
\item[(b)] The functor $-\otimes_B T$ is right exact.  Moreover, if $0\rightarrow L\rightarrow M\rightarrow N\rightarrow 0$ is a short exact sequence in $\mod B$, with $L,M,N\in \Y$, then the induced sequence $0\rightarrow L\otimes_B T\rightarrow M\otimes_B T\rightarrow N\otimes_B T\rightarrow 0$ is exact in $\mod A$.
\item[(c)] The functor $\Ext^1_A(T, -)$ is right exact.  Moreover, if $0\rightarrow L\rightarrow M\rightarrow N\rightarrow 0$ is a short exact sequence in $\mod A$, with $L,M,N\in \F$, then the induced sequence $0\rightarrow \Ext^1_A(T,L)\rightarrow \Ext^1_A(T,M)\rightarrow \Ext^1_A(T,N)\rightarrow 0$ is exact in $\mod B$.
\item[(d)] The functor $\Tor_1^B(-, T)$ is left exact.  Moreover, if $0\rightarrow L\rightarrow M\rightarrow N\rightarrow 0$ is a short exact sequence in $\mod B$, with $L,M,N\in \X$, then the induced sequence $0\rightarrow \Tor_1^B(L, T)\rightarrow \Tor_1^B(M, T)\rightarrow \Tor_1^B(N, T)\rightarrow 0$ is exact in $\mod A$.
\end{enumerate}
\end{remark}

The next result completes the above remark and characterizes the families $\ov{\F}$, $\und{\T}$, $\ov{\Y}$ and $\und{\X}$.

\begin{proposition}\label{Cor 3.3 and Cor 3.5}
Let $T$ be a tilting $A$-module.
\begin{enumerate}
\item[(a)] If $M\in \mod A$, then $M\in \und{\T}$ if and only if, for each epimorphism $f:M\rightarrow N$ in $\mod A$, the induced sequence $0\rightarrow \Hom_A(T, X)\rightarrow \Hom_A(T, N)\rightarrow \Ext_A^1(T, \Ker f)\rightarrow 0$ is exact in $\mod B$.
\item[(b)] If $M\in\mod B$, then $M\in \ov{\Y}$ if and only if, for each monomorphism $f:M\rightarrow X$ in $\mod B$, then induced sequence $0\rightarrow \Tor^B_1(\Coker f, T)\rightarrow M\otimes_B T\rightarrow X\otimes_B T\rightarrow 0$ is exact in $\mod A$.
\item[(c)] If $M\in \mod A$, then $M \in \ov{\F}$ if and only if, for each monomorphism $f:M\rightarrow X$ in $\mod A$, the induced sequence $0\rightarrow \Hom_A(T, \Coker f)\rightarrow \Ext_A^1(T, M)\rightarrow \Ext_A^1(T, X)\rightarrow 0$ is exact in $\mod B$.
\item[(d)] If $M\in\mod B$, then $M\in \und{\X}$ if and only if, for each epimorphism $f:X\rightarrow N$ in $\mod B$, the induced sequence $0\rightarrow \Tor_1^B(X, T)\rightarrow \Tor_1^B(N, T)\rightarrow \Ker f\otimes_B T\rightarrow 0$ is exact in $\mod A$.
\end{enumerate}
\end{proposition}
\begin{proof}
We only prove (a); the other proofs can be obtained by analogue approaches, sometimes using $-\otimes_B T$ instead of $\Hom_A(T,-)$. 
For the necessity, suppose that $M\in \und{\T}$ and $f:M\rightarrow N$ is an epimorphism in $\mod A$. If $f$ is an isomorphism, the sequence is clearly exact. Else, since $M\in \und{\T}$, it follows from (\ref{Prop 3.2 and Cor 3.14})(c) that $\Ker f\in \F$. The statement then follows from the fact that $\Hom_A(T, -)\mid_{\F}=0$ and $\Ext^1_A(T,-)\mid_{\T}=0$.

Conversely, suppose that for each epimorphism $f:M\rightarrow N$ $\mod A$, the sequence $0\rightarrow \Hom_A(T, X)\rightarrow \Hom_A(T, N)\rightarrow \Ext_A^1(T, \Ker f)\rightarrow 0$ is exact in $\mod B$. Suppose that $L$ is a proper nonzero submodule of $M$.  We have an epimorphism $f:M\rightarrow M/L$ with $\Ker f=L$, and it follows form the exactness of the sequence that $\Hom_A(T, L)=0$ and $\Ext^1_A(T,M)=0$. Consequently, $M\in\T$ and $L\in \F$ by (\ref{Prop 3.2 and Cor 3.14})(c).
\end{proof}

\begin{proposition}\label{Prop 3.6}
Let $T$ be a tilting $A$-module. 
\begin{enumerate}
\item[(a)] $\Hom_A(T, \und{\T})=\ov{\Y}$.
\item[(b)] $\ov{\Y}\otimes_BT=\und{\T}$.
\item[(c)] $\Ext_A^1(T,\ov{\F})=\und{\X}$.
\item[(d)] $\Tor^B_1(\und{\X}, T)=\ov{\F}$. 
\end{enumerate}
\end{proposition}
\begin{proof}
We only prove (a); the other proofs can be obtained by analogue arguments. Let $X\in\und{\T}$.  Set $Y=\Hom_A(T,-)$ and suppose that $Y\notin \ov{\Y}$. Then there exists a proper nonzero submodule $L$ of $Y$ such that $Y/L\in\Y$.  Since $\Y$ is closed under submodules, we get a short exact sequence $0\rightarrow L\rightarrow Y \rightarrow Y/L\rightarrow 0$ in $\Y$.  By (\ref{Rem 3.4})(b), the induced sequence $0\rightarrow L\otimes_B T\rightarrow Y\otimes_B T\rightarrow Y/L\otimes_B T\rightarrow 0$ is exact in $\mod B$. However, $Y\otimes_B T=\Hom_A(T, X)\otimes_B T\cong X$ by the Brenner-Butler theorem.  Since $X\in\und{\T}$, we get $L\otimes_B T=0$, and so $L=0$, a contradiction.  Therefore $Y=\Hom_A(T,X)\in\ov{\Y}$.
\end{proof}


We are now in position to demonstrate the main result of this section.  This result shows that the smallest families of modules that filter $\T$, $\F$, $\X$ and $\Y$ are totally compatible with the Brenner-Butler theorem.

\begin{theorem}\label{Prop 3.8}
Let $T$ be a tilting $A$-module.
\begin{enumerate}
\item[(a)] Let $M\in\T$.
\begin{enumerate}
\item[(i)] If $0\subset M_1 \subset \cdots \subset M_{n-1}\subset M
$
is a $\und{\T}$-filtration of $M$, then the sequence 
\[
0\subset \Hom_A(T,M_1) \subset \cdots \subset \Hom_A(T,M_{n-1})\subset \Hom_A(T,M)
\]
induced by $\Hom_A(T, -)$ is a $\ov{\Y}$-filtration of $\Hom_A(T,M)$.
\item[(ii)] If 
$
M\rightarrow M_{n-1} \rightarrow \cdots \rightarrow M_1 \rightarrow 0
$
is a $\und{\T}$-cofiltration of $M$, then the sequence
\[
\Hom_A(T,M)\rightarrow\Hom_A(T,M_{n-1})\rightarrow \cdots \rightarrow\Hom_A(T,M_1) \rightarrow 0
\]
induced by $\Hom_A(T, -)$
is a $\ov{\Y}$-cofiltration of $\Hom_A(T,M)$.
\end{enumerate}
\item[(b)] Let $M\in\Y$.
\begin{enumerate}
\item[(i)] If $
0\subset M_1 \subset \cdots \subset M_{n-1}\subset M
$
is a $\ov{\Y}$-filtration of $M$, then the sequence
\[
0\subset M_1\otimes_B T \subset \cdots \subset M_{n-1}\otimes_B T\subset M\otimes_B T
\]
induced by $-\otimes_B T$
is a $\und{\T}$-filtration of $M\otimes_B T$.
\item[(ii)] If 
$
M\rightarrow M_{n-1} \rightarrow \cdots \rightarrow M_1 \rightarrow 0
$
is a $\ov{\Y}$-cofiltration of $M$, then the sequence
\[
M\otimes_B T\rightarrow M_{n-1}\otimes_B T \rightarrow \cdots \rightarrow M_1\otimes_B T \rightarrow 0
\]
induced by $-\otimes_B T$
is a $\und{\T}$-cofiltration of $M\otimes_B T$.
\end{enumerate}
\item[(c)] Let $M\in\F$.
\begin{enumerate}
\item[(i)] If $
0\subset M_1 \subset \cdots \subset M_{n-1}\subset M
$
is a $\ov{\F}$-filtration of $M$, then the sequence
\[
0\subset \Ext^1(T,M_1) \subset \cdots \subset \Ext^1(T,M_{n-1})\subset \Ext^1(T,M)
\]
induced by $\Ext^1(T,-)$
is a $\und{\X}$-filtration of $\Ext^1(T,M)$.
\item[(ii)] If 
$
M\rightarrow M_{n-1}\rightarrow \cdots\rightarrow M_1 \rightarrow 0
$
is a $\ov{\F}$-cofiltration of $M$, then the sequence
\[
\Ext^1(T,M) \rightarrow \Ext^1(T,M_{n-1}) \rightarrow\cdots \rightarrow\Ext^1(T,M_1)\rightarrow 0
\]
induced by $\Ext^1(T,-)$
is a $\und{\X}$-cofiltration of $\Ext^1(T,M)$.
\end{enumerate}
\item[(d)] Let $M\in\X$.
\begin{enumerate}
\item[(i)] If $
0\subset M_1 \subset \cdots \subset M_{n-1}\subset M
$
is a $\und{\X}$-filtration of $M$, then the sequence
\[
0\subset \Tor^B_1(M_1, T) \subset \cdots \subset \Tor^B_1(M_{n-1}, T) \subset \Tor^B_1(M, T)
\]
induced by $\Tor^B_1(-,T)$
is a $\ov{\F}$-filtration of $\Tor^B_1(M, T)$.
\item[(ii)] If 
$
M\rightarrow M_{n-1} \rightarrow \cdots \rightarrow M_1 \rightarrow 0
$
is a $\und{\X}$-cofiltration of $M$, then the sequence
\[
\Tor^B_1(M, T)\rightarrow\Tor^B_1(M_{n-1}, T) \rightarrow\cdots\rightarrow \Tor^B_1(M_1, T) \rightarrow 0
\]
induced by $\Tor^B_1(-,T)$
is a $\ov{\F}$-cofiltration of $\Tor^B_1(M, T)$.
\end{enumerate}
\end{enumerate}
\end{theorem}
\begin{proof}
We only prove (a)(i); the other proofs can be obtained by analogue arguments.
Suppose that $0=M_0\subset M_1 \subset \cdots \subset M_{n-1}\subset M_n=M
$
is a $\und{\T}$-filtration of $M$.  It then follows from (\ref{Rem 3.4})(a) that $0\rightarrow \Hom_A(T,M_{i-1})\rightarrow \Hom_A(T,M_i)\rightarrow \Hom_A(T,M_i/M_{i-1})\rightarrow 0$ is exact for all $i=1, 2, \dots, n$.  Moreover, $\Hom_A(T,M_i/M_{i-1})\in \ov{\Y}$ for all $i=1, 2, \dots, n$ by (\ref{Prop 3.6})(a).  This shows the claim.

\end{proof}

\begin{remark}\label{Cor 3.9}
Let $T$ be a tilting $A$-module, and let $M\in \T$. By applying consecutively the constructions explained in (\ref{Prop 1.7}) and (\ref{Prop 3.8}), one can construct $\ov{\Y}$-cofiltrations of $\Hom_A(T,M)$ given a $\und{\T}$-filtration of $M$. Actually, depending on whether the construction (\ref{Prop 1.7}) or (\ref{Prop 3.8}) is applied first, one might expect to obtain different $\ov{\Y}$-cofiltrations of $\Hom_A(T,M)$. Similarly, one can possibly construct two different $\ov{\Y}$-filtrations of $\Hom_A(T,M)$ given a $\und{\T}$-cofiltration of $M$. 

However, it is not difficult to verify, using (\ref{Rem 3.4}), the First Isomorphism Theorem and the left exactness of $\Hom_A(T, -)$, that in both cases the two constructions commute up to isomorphisms.  More precisely, the $B$-modules appearing in the $\ov{\Y}$-cofiltrations (or $\ov{\Y}$-filtrations) of $\Hom_A(T,M)$ are the same, up to isomorphisms. 

Similarly, it is possible to show that the constructions (\ref{Prop 1.7}) and (\ref{Prop 3.8}) commute, up to isomorphisms, for all quasi-inverse equivalences in the Brenner-Butler theorem. The easy but tedious verification is left to the reader.
\end{remark}

\subsection{Modules with (co)standard filtrations}

Let $A$ be an artinian algebra, and let $\theta$ be a subset of $\mod A$ that is closed under isomorphisms.  It is easy to verify that $\F(\theta)$ is closed under extensions.  Moreover, $\ov{\F(\theta)}\cup \und{\F(\theta)}\subseteq \theta$ when $\theta$ consists of nonzero modules.  Indeed, suppose that $M\in\F(\theta)$ and that $0\subset M_1\subset M_2 \subset\cdots\subset M_n=M$ is a $\theta$-filtration of $M$. In particular $M/M_{n-1}\in\theta$.  If $M\in\und{\F(\theta)}$, then $M_{n-1}=0$, while if $M\in\ov{\F(\theta)}$, then $M/M_{n-1}=M$, that is $M_{n-1}=0$.  So $M\in\theta$.

However, in general, we do not have $\theta\subseteq \ov{\F(\theta)}\cup \und{\F(\theta)}$. For instance, if $A$ is a non-semisimple algebra of finite global dimension, and we take $\theta=\mod A$, then $\F(\theta)=\mod A$ but $\ov{\F(\theta)}=\und{\F(\theta)}$ consists of all simple $A$-modules by (\ref{Lem 3.11}) and (\ref{Rem 1.1(a)(b)}), while $\mod A$ has non-simple modules, so $\ov{\F(\theta)}\cup\und{\F(\theta)}\neq \theta$.

In this section, given a subset $\theta$ of $\mod A$ we exhibit a sufficient condition to have $\ov{\F(\theta)}\cup\und{\F(\theta)}=\theta$ when $\F(\theta)$ is closed under kernels of epimorphisms or closed under cokernels of monomorphisms.  
We then apply this result to the modules having standard or costandard filtrations. Recall that these modules play an important role in representation theory of algebras, in particular for the construction of standardly stratified and quasi-hereditary algebras.  These latter algebras were used by Iyama to show that the representation dimension of any artin algebra is always finite, see \cite{I03}.

Before stating the first result, recall from (\ref{Rem 1.1(a)(b)}) that $\ov{\F(\theta)}\cup\und{\F(\theta)}=\ov{\F(\theta)}$ when $\F(\theta)$ is closed under kernels of epimorphisms, and $\ov{\F(\theta)}\cup\und{\F(\theta)}=\und{\F(\theta)}$ when $\F(\theta)$ is closed under cokernels of monomorphisms.

\begin{proposition}\label{Prop sufficient}
Let $A$ be an artinian algebra, and let $\theta$ be a subset of $\mod A$ that does not contain the zero module and is closed under isomorphisms. 
\begin{enumerate}
\item[(a)] Suppose that $\F(\theta)$ is closed under kernels of epimorphisms.  If every epimorphism $\theta_1\rightarrow \theta_2$, with $\theta_1, \theta_2\in\theta$, is an isomorphism, then $\ov{\F(\theta)}=\theta$.
\item[(b)] Suppose that $\F(\theta)$ is closed under cokernels of monomorphisms.  If every monomorphism $\theta_1\rightarrow \theta_2$, with $\theta_1, \theta_2\in\theta$, is an isomorphism, then $\und{\F(\theta)}=\theta$.
\end{enumerate}
\end{proposition}
\begin{proof}
We only prove (a); the proof of (b) is dual. First, as $\F(\theta)$ is closed under extensions and kernels of epimorphisms, it follows from (\ref{Thm 1.5}) that $\ov{\F(\theta)}$ is the smallest class of modules that filters $\F(\theta)$. Therefore $\ov{\F(\theta)}\subset \theta$.  

To show the reverse inclusion, suppose that $\theta_1\in\theta$.  As $\theta\subseteq\F(\theta)$, there exists an epimorphism $f:\theta_1\rightarrow X$, with $X\in \ov{\F(\theta)}$, see (\ref{Lem 1.3}). But $X\in\F(\theta)$, so there exists an epimorphism $g:X\rightarrow \theta_2$, with $\theta_2\in \theta$.  By hypothesis on $\theta$, the composition $gf:\theta_1\rightarrow X\rightarrow \theta_2$ is an isomorphism.  Clearly, $g\neq 0$.  It then follows from the facts that $X\in \ov{\F(\theta)}$ and $\theta_2\in\F(\theta)$ that $g$ is an isomorphism, showing that $X\in\theta$.  
\end{proof}

We finish with a discussion on the modules with (proper) standard filtrations or (proper) costandard filtrations that appear in the context of quasi-hereditary algebras and their generalizations.

\begin{example}
Let $A$ be an artin algebra, and $P_1, P_2, \dots, P_n$ be a fixed ordering of the indecomposable projective $A$-modules, with corresponding simple tops $S_1, S_2, \dots, S_n$.  Let $\Delta_i$ be the largest factor of $P_i$ with no composition factor $S_j$, with $j>i$.  Denote by ${\Delta_i}^*$ the largest factor of $\Delta_i$ where $S_i$ occurs only once as a composition factor.  Finally, let $\Delta=\{\Delta_1, \Delta_2, \dots, \Delta_n\}$ and ${\Delta}^*=\{{\Delta_1}^*, {\Delta_2}^*, \dots, {\Delta_n}^*\}$. The objects in $\F(\Delta)$ are those with \textbf{standard filtrations}, and the objects in $\F(\Delta^*)$ are those with \textbf{proper standard filtrations}.

Dually, let $I_1, I_2, \dots, I_n$ be the indecomposable injective modules, with soc$I_j\cong S_j$.  Let $\nabla_i$ be the largest submodule of $I_i$ with no composition factor $S_j$ with $j>i$, and ${\nabla_i}^*$ be the largest submodule of $\nabla_i$ with $S_i$ occurring only once as composition factor.  Finally, let $\nabla=\{\nabla_1, \nabla_2, \dots, \nabla_n\}$ and ${\nabla}^*= \{{\nabla_1}^*, {\nabla_2}^*, \dots, {\nabla_n}^*\}$. The objects in $\F(\nabla)$ are those with \textbf{costandard filtrations}, and the objects in $\F(\nabla^*)$ are those with \textbf{proper costandard filtrations}.

It is known that $\F(\Delta)$ and $\F({\Delta}^*)$ are closed under direct summands and kernels of epimorphisms in $\mod A$ and that $\F(\nabla)$ and $\F({\nabla}^*)$ are closed under direct summands and cokernels of monomorphisms in $\mod A$, see \cite{ADL98, DR92, R91}. Moreover, if $\Delta_i$ and $\Delta_j$ are two modules in $\Delta$, then $\top\Delta_i=S_i$ and $\top\Delta_j=S_j$ by construction. Therefore, every epimorphism from $\Delta_i$ to $\Delta_j$ is an isomorphism. The same argument holds for $\Delta^*$. Dually, every module in $\nabla$ or $\nabla^*$ has a simple socle, so any monomorphism between two such modules in $\nabla$ or $\nabla^*$ is an isomorphism. 

Consequently, it follows from (\ref{Prop sufficient}) that $\ov{\F(\Delta)}=\Delta$ and $\ov{\F({\Delta}^*)}={\Delta}^*$ and $\und{\F(\nabla)}=\nabla$ and $\und{\F({\nabla}^*)}={\nabla}^*$.
\end{example}



%

\end{document}